\documentclass[a4paper,12pt]{amsart}
\usepackage[a4paper,margin=1.2in]{geometry}
\usepackage{amsmath,amsthm,amssymb,mathrsfs}
\usepackage{mathtools}
\usepackage{enumitem}
\usepackage{hyperref}
\usepackage{tikz}
\usetikzlibrary{arrows.meta}
\usepackage[expansion=false,nopatch=footnote]{microtype}

\hypersetup{
  colorlinks=false,
  linkcolor=black,
  citecolor=black,
  urlcolor=black,
  pdftitle={Two characters on one punctured Riemann surface},
  pdfauthor={Pradip Kumar},
}

\newcommand{\C}{\mathbb{C}}
\newcommand{\R}{\mathbb{R}}
\newcommand{\Z}{\mathbb{Z}}

\DeclareMathOperator{\Ext}{Ext}

\newcommand{\Hom}{\operatorname{Hom}}
\newcommand{\Per}{\operatorname{Per}}

\newcommand{\id}{\mathrm{id}}

\newcommand{\Mgn}{\mathcal{M}_{g,n}}
\newcommand{\Om}{\Omega}

\theoremstyle{plain}
\newtheorem{theorem}{Theorem}[section]
\newtheorem{proposition}[theorem]{Proposition}
\newtheorem{lemma}[theorem]{Lemma}
\newtheorem{corollary}[theorem]{Corollary}

\theoremstyle{definition}
\newtheorem{definition}[theorem]{Definition}

\newtheorem{Hypotheses}[theorem]{Hypotheses}

\title[Two characters on one surface]{Two characters on one punctured Riemann surface}

\author{Pradip Kumar}
\address{Department of Mathematics, Shiv Nadar University, Delhi NCR, India}
\email{pradip.kumar@snu.edu.in}
\date{}
\subjclass[2020]{57M50, 30F30, 32G15, 30F60}
\keywords{meromorphic differentials, period map, punctured Riemann surfaces, extremal length, Teichm\"uller theory, minimal surfaces}

\begin{document}

\begin{abstract}
We develop an abstract framework for coupled period--realization of meromorphic $1$--forms on punctured Riemann surfaces.
A configuration datum $C$ gives the combinatorics and determines a restricted character domain
$\Delta_C\subset\Hom(\Gamma_{g,n},\C)^2$ with a scale--fixed slice $\Delta_C^{\mathrm{sc}}$. 
Assuming Teichm\"uller--regularity, degeneration detection, and pushability, we prove that there is point in $\Delta_C^{\mathrm{sc}}$  which corresponds to a surface carrying two meromorphic differentials realizing any prescribed restricted pair.
This abstracts the Weber--Wolf extremal--length minimization method \cite{WeberWolfGAFA,WeberWolfAnnals}.
\end{abstract}

\maketitle

\tableofcontents
\section{Introduction}
Fix integers $g\ge 0$ and $n\ge 1$.
Let $S_{g,\, n}$ be an oriented topological surface of genus $g$ with $n$ punctures, and set
\[
\Gamma_{g,\, n}:=\mathrm{H}_1(S_{g,\, n};\Z). 
\]
When $n=0$ we write $S_g$ and $\Gamma_g:=\mathrm{H}_1(S_g;\Z)$. A \emph{character} is a group homomorphism
$\chi:\Gamma_{g,\, n}\longrightarrow\C$.

A (marked) punctured Riemann surface of type $(g , \, n)$ is a pair $(X, \, P)$ consisting of a compact Riemann surface $X$ of genus
$g$ and an ordered set $P=\{p_1,\, \dots,\, p_n\}\subset X$ of distinct points, together with a marking
$S_{g,\, n}\cong X\setminus P$.
We write $\Mgn$ for the corresponding moduli space of marked punctured Riemann surfaces.

For such $(X,\, P)$ we denote by $\Om(X)$ the vector space of meromorphic $1$--forms on $X$ that are holomorphic on
$X\setminus P$.
Any $\omega\in\Om(X)$ has a \emph{period character}
\[
\chi_\omega:\Gamma_{g,\, n}\longrightarrow \C,\qquad \gamma\longmapsto \int_\gamma \omega,
\]
where $\gamma$ is viewed on $X\setminus P$ via the marking.
Let $\Omega\mathcal M_{g,\, n}$ denote the total space of the bundle of such differentials over $\Mgn$; the \emph{period map}
is
\[
\Per:\Omega\mathcal M_{g,\, n}\longrightarrow \Hom(\Gamma_{g,\, n},\, \C),\qquad (X,\, \omega)\longmapsto \chi_\omega.
\]

\subsection{Realization of a single character}
Let $X$ be a compact Riemann surface of genus $g\ge 1$.
Any holomorphic $1$--form $\omega\in H^0(X,\, K_X)$ determines its period character on $\Gamma_g$.
One may ask which characters occur as $\chi_\omega$ (allowing $X$ to vary in moduli).
When $g\ge 2$, the answer is governed by two intrinsic obstructions discovered by Haupt.
To state them, fix a symplectic basis $\{\alpha_1,\, \beta_1,\, \dots,\, \alpha_g,\, \beta_g\}$ of $\Gamma_g$ and define
\[
\mathrm{vol}(\chi)\ :=\ \sum_{i=1}^g \mathrm{Im}\!\bigl(\chi(\alpha_i)\, \, \overline{\chi(\beta_i)}\bigr),
\qquad \chi\in\Hom(\Gamma_g,\, \C).
\]
\begin{theorem}[Haupt {\cite{Haupt}}]\label{thm:intro-haupt}
Assume $g\ge 2$.  A character $\chi\in\Hom(\Gamma_g,\, \C)$ lies in the image of the period map
\[
\Per:\ \Omega\mathcal{M}_g \longrightarrow \Hom(\Gamma_g,\, \C)
\]
if and only if
\begin{enumerate}[label=\textnormal{(\roman*)}, itemsep=2pt, topsep=2pt]
\item $\mathrm{vol}(\chi)>0$, and
\item if $\chi(\Gamma_g)=\Lambda\subset\C$ is a lattice, then
\[
\mathrm{vol}(\chi)\ \ge\ 2\,  \, \mathrm{Area}(\C/\Lambda).
\]
\end{enumerate}
For $g=1$, a character $\chi\in\Hom(\Gamma_1,\, \C)$ arises from a holomorphic abelian differential if and only if
$\chi(\Gamma_1)$ is a rank--$2$ lattice and $\mathrm{vol}(\chi)>0$.
\end{theorem}

For punctured surfaces $X\setminus P$ the relevant objects are meromorphic differentials holomorphic on $X\setminus P$.
In this noncompact setting, one gains substantial flexibility after allowing the surface to vary in $\Mgn$:

\begin{theorem}[Chenakkod--Faraco--Gupta {\cite[Theorem~A]{CFG}}]\label{thm:intro-cfg}
If $n\ge 1$, the period map
\[
\Per:\ \Omega\mathcal{M}_{g,\, n}\longrightarrow \Hom(\Gamma_{g,\, n},\, \C)
\]
is surjective.  Equivalently, every $\chi\in\Hom(\Gamma_{g,\, n},\, \C)$ is the period character of a meromorphic differential on
some punctured Riemann surface $X\in\mathcal{M}_{g,\, n}$.
\end{theorem}

\subsection{Coupled restricted period data and the Weber--Wolf mechanism}\label{subsec:coupled}
In geometric constructions one mostly encounters a pair of compatible characters. The period data are constrained by combinatorics and by symmetry: one prescribes finitely many periods and imposes linear
``closure'' relations together with ray/positivity conditions.
In the orthodisk/Weierstrass approach to minimal surfaces, for example, one starts with Weierstrass data $(G,\, dh)$ on a
punctured surface and considers the two meromorphic differentials
\[
\omega_{\mathrm{I}} := G\, dh,\qquad \omega_{\mathrm{II}} := G^{-1}\,dh .
\]
The global period conditions for the immersion force the periods of $\omega_{\mathrm{I}}$ and $\omega_{\mathrm{II}}$ to satisfy rigid relations.
One is thus led to a \emph{simultaneous} realization problem for a restricted pair of period characters.

\medskip
\noindent\textbf{Problem.}\; Given a restricted pair of characters $(\chi_{\mathrm{I}},\chi_{\mathrm{II}})$, can one find a single punctured Riemann
surface $X$ and two meromorphic $1$--forms $\omega_{\mathrm{I}},\omega_{\mathrm{II}}\in\Om(X)$ such that
$\chi_{\omega_{\mathrm{I}}}=\chi_{\mathrm{I}}$ and $\chi_{\omega_{\mathrm{II}}}=\chi_{\mathrm{II}}$?
\medskip

The difficulty is not the separate realization of each character (Theorem~\ref{thm:intro-cfg} provides pointwise existence), but the
\emph{coupling with restricted order of zero or pole}: in concrete models each character typically determines its own conformal structure, and one must arrange that the two
conformal structures coincide.

Weber and Wolf introduced a Teichm\"uller--theoretic minimization strategy to solve this kind of coupled period problem in their orthodisk
approach to minimal surfaces~\cite{WeberWolfGAFA,WeberWolfAnnals}.
Starting from the pair of $1$--forms $\omega_{\mathrm{I}}=G\,dh$ and $\omega_{\mathrm{II}}=G^{-1}dh$, they constructed two  flat structures
(or planar domains) whose conformal types must agree in order for the Weierstrass data to live on a single Riemann surface.
The key idea is to compare extremal lengths of appropriately paired curve families on the two sides and to assemble the resulting  discrepancies into a nonnegative \emph{height} function.
A global minimizer of the height cannot be a positive local minimum, because one can ``push'' the configuration to decrease the height.
Consequently, at a global minimizer the height must vanish, forcing the two conformal structures to coincide and producing a single punctured
surface realizing both differentials.

Variants of this extremal--length minimization philosophy have since appeared in other Weierstrass--type constructions, for instance in the
study of maxfaces~\cite{BardhanBiswasKumarMaxfaces} and, more recently, in the construction of higher genus Angel surfaces~\cite{BardhanBiswasFujimoriKumarAngel}.
Motivated by the apparent portability of the argument, the present paper abstracts the Weber--Wolf mechanism: we isolate a list of analytic
inputs---formulated purely in terms of extremal length on a parameter space---under which the corresponding height function must vanish.
The purpose of this abstraction is twofold.
First, it separates the essentially Teichm\"uller--theoretic argument from the geometry that produces a given parameter space.
Second, it packages the method in a form that can be transported verbatim to other coupled period problems.

\subsection{Motivating one--parameter model}\label{subsec:motivating-example}
We briefly illustrate the mechanism in a simple one--parameter family.
Fix the oriented surface $S_{2,2}$ of genus $2$ with two punctures and write $\Gamma_{2,2}=\mathrm{H}_1(S_{2,2};\Z)$.
Choose a symplectic basis $\alpha_1,\beta_1,\alpha_2,\beta_2$ for the genus--two part and let $\delta_1,\delta_2$ be small
positively oriented loops around the punctures (so $\delta_1+\delta_2=0$ in homology).

Consider pairs of characters $(\chi_{\mathrm{I}},\chi_{\mathrm{II}})\in\Hom(\Gamma_{2,2},\C)^2$ subject to three types of restrictions:
\begin{enumerate}[label=\textnormal{(\roman*)}, itemsep=2pt, topsep=2pt]
\item a common puncture period $\chi_{\mathrm{I}}(\delta_1)=\chi_{\mathrm{II}}(\delta_1)=2\pi i\,r$ with $r\in\R\setminus\{0\}$;
\item ray conditions placing selected periods on $\R_{>0}$ and $i\R_{>0}$, with ``real/imaginary'' directions exchanged
between $\chi_{\mathrm{I}}$ and $\chi_{\mathrm{II}}$;
\item linear closure relations $\chi_{\mathrm{I}}(\alpha_1)=\chi_{\mathrm{I}}(\alpha_2)$ and $\chi_{\mathrm{II}}(\beta_1)=\chi_{\mathrm{II}}(\beta_2)$.
\end{enumerate}
These conditions cut out a restricted domain $\Delta\subset\Hom(\Gamma_{2,2},\C)^2$.
Fix $a,b>0$ and set, for $t>0$,
\[
\chi_{\mathrm{I}}(t)(\alpha_1)=\chi_{\mathrm{I}}(t)(\alpha_2)=a,\qquad
\chi_{\mathrm{I}}(t)(\beta_1)=\chi_{\mathrm{I}}(t)(\beta_2)= i\,t,\qquad
\chi_{\mathrm{I}}(t)(\delta_1)=2\pi i\,r,
\]
and
\[
\chi_{\mathrm{II}}(t)(\alpha_1)=\chi_{\mathrm{II}}(t)(\alpha_2)= i\,t,\qquad
\chi_{\mathrm{II}}(t)(\beta_1)=\chi_{\mathrm{II}}(t)(\beta_2)= b,\qquad
\chi_{\mathrm{II}}(t)(\delta_1)=2\pi i\,r.
\]
Then $t\longmapsto(\chi_{\mathrm{I}}(t),\chi_{\mathrm{II}}(t))$ parametrizes a scale--fixed slice $\Delta^{\mathrm{sc}}\subset\Delta$.

In applications one is given (or constructs) separate realization families $t\longmapsto (X_{\mathrm{I, II}}(t),\omega_{\mathrm{I, II}}(t))$ with
$\chi_{\omega_{\mathrm{I, II}}(t)}=\chi_{\mathrm{I, II}}(t)$.
The restrictions above canonically pair curve classes $\alpha_i\leftrightarrow \beta_i$.
Comparing extremal lengths along this pairing, one defines mismatch functions
\[
m_i(t):=\log\frac{\Ext_{X_{\mathrm{I}}(t)}(\alpha_i)}{\Ext_{X_{\mathrm{II}}(t)}(\beta_i)},\qquad i=1,2,
\]
and a height
\[
H(t):=m_1(t)^2+m_2(t)^2\ge 0 .
\]
Degeneration typically forces $H(t)\longrightarrow\infty$ as $t\longrightarrow 0$ or $t\longrightarrow\infty$, so $H$ attains a minimum.
The key point of the Weber--Wolf argument is that any positive local minimum can be decreased by a controlled deformation (``pushability'');
hence a global minimizer must satisfy $H(t_0)=0$.
Once one has enough paired curve families to separate Teichm\"uller space, the condition $H(t_0)=0$ forces
$X_{\mathrm{I}}(t_0)\cong X_{\mathrm{II}}(t_0)$ and gives a single punctured surface carrying both differentials.

This article explains such a mechanism in a general abstract setting, with the aim of making it directly reusable in other coupled period problems.

\medskip
\noindent\textbf{Organization of the paper.}
Section~\ref{sec:terminology} sets up the abstract terminology and states the main results.
Section~\ref{sec:config_space_new} defines restricted character domains and introduces the extremal--length data and the  hypotheses.
Under these axioms, Section~\ref{sec:proof_main} proves the existence of a reflexive point (Theorem~\ref{thm:main}) and deduces the simultaneous
realization statement (Theorem~\ref{thm:twoforms}).
Finally, Section~\ref{sec:examples} provides examples illustrating both success and failure of the hypotheses.

\section{Terminology and statement of the main theorem}\label{sec:terminology}
In the geometric problems motivating this paper, the period data is not arbitrary.
Typically one prescribes two period characters (corresponding to two geometric realizations, denoted $\mathrm{I}$ and $\mathrm{II}$),
but only after imposing a finite system of constraints
and a rule pairing the labels on the $\mathrm{I}$--side with those on the $\mathrm{II}$--side.
These requirements are intrinsic to the geometry,
so they are natural rather than artificial restrictions once one insists on the intended geometric realization.

We introduce the notion of a \emph{configuration datum} to encode these combinatorial constraints and the pairing rule, and then we define the associated \emph{restricted character domain} $\Delta_C$ as the subset of $\Hom(\Gamma,\C)^2$ cut out by these constraints.

\begin{definition}[Configuration datum]\label{def:config}
A \emph{configuration datum} is a $5$--tuple
\[
C=(E,\iota,R,\tau,\sigma)
\]
consisting of:
\begin{enumerate}
\item a finite set $E$ (the set of \emph{edge labels});
\item a map $\iota:E\longrightarrow\Gamma$ such that the subgroup $\langle\iota(E)\rangle\subset\Gamma$ generated by $\iota(E)$
equals $\Gamma$;
\item a subgroup $R=\ker(\iota_*:\Z^E\longrightarrow\Gamma)$ generated by finitely many elements
$r^{(1)},\dots,r^{(m)}\in\Z^E$ (the \emph{closure relations});
\item a map $\tau:E\longrightarrow\{h,v\}$ (the \emph{type map}), interpreted as ``horizontal'' $(h)$ or ``vertical'' $(v)$;
\item a bijection $\sigma:E\longrightarrow E$ (the \emph{edge correspondence}) such that
\[
\tau\circ\sigma=\tau .
\]
\end{enumerate}
\end{definition}

Associated to the discrete configuration datum \(C\) we define a class of \emph{admissible period assignments}. The datum \(C\) is purely combinatorial; from it we obtain a subset
\[
\Delta_C \subset \Hom(\Gamma,\C)^2
\] by imposing conditions on \((\chi^{\mathrm{I}},\chi^{\mathrm{II}})\): namely,
\((\chi^{\mathrm{I}},\chi^{\mathrm{II}})\in\Delta_C\) if and only if
\begin{enumerate}[itemsep=2pt, topsep=2pt]
\item the values of \(\chi^{\mathrm{I}}\) and \(\chi^{\mathrm{II}}\) on the labeled cycles satisfy all linear
relations in \(R\);
\item the specified values lie in the rays/half-planes dictated by \(\tau\); and
\item whenever two labels are paired by \(\sigma\), the corresponding periods of \(\chi^{\mathrm{I}}\) and
\(\chi^{\mathrm{II}}\) satisfy the matching rule prescribed by \(C\).
\end{enumerate}
In other words, \(\Delta_C\) is obtained from \(C\) by translating each piece of discrete geometric data into
an explicit algebraic or inequality constraint on the pair of characters.
Since overall scaling of period data does not change conformal type in the geometric models of interest,
we also fix scale and work on the slice (Definition~\ref{def:DeltaCsc})
\[
\Delta_C^{\mathrm{sc}}\subset \Delta_C.
\]

\begin{definition}[Admissible curve set]\label{def:AC}
An \emph{admissible curve set} for $C$ is a finite collection $\mathcal A_C$ of isotopy classes of essential
simple closed curves and properly embedded essential arcs on $S_{g,n}$,
together with a designated involution
\[
\sigma_\ast:\mathcal A_C\longrightarrow\mathcal A_C,\qquad \gamma\longmapsto \gamma^\sigma,
\]
encoding which curve family on the $\mathrm{II}$--side is paired with a given curve family on the $\mathrm{I}$--side.
\end{definition}

\medskip

For the analytic part of the paper, we fix an admissible curve set $\mathcal A_C$ and choose
continuous extremal--length assignments $\Ext_{\mathrm{I}},\,\Ext_{\mathrm{II}}$ in the sense of
Definition~\ref{def:ExtAssign}.
Concretely, for each curve family $\gamma\in\mathcal A_C$ we are given two positive functions
\[
u\longmapsto \Ext_{\mathrm{I}}(u;\gamma),\qquad
u\longmapsto \Ext_{\mathrm{II}}(u;\gamma)
\]
on $\Delta_C^{\mathrm{sc}}$. For $\gamma\in\mathcal A_C$, and \(u\in\Delta_C^{\mathrm{sc}}\) define the \emph{mismatch}
\[
m_\gamma(u)\ :=\ \log\frac{\Ext_{\mathrm{I}}(u;\gamma)}{\Ext_{\mathrm{II}}(u;\gamma^\sigma)},
\]
and  the \emph{height}
\begin{equation}\label{eq:height-def}
H(u)\ :=\ \sum_{\gamma\in\mathcal A_C} m_\gamma(u)^2.
\end{equation}
We call $u\in\Delta_C^{\mathrm{sc}}$ \emph{reflexive} if $H(u)=0$, equivalently if
\[
\Ext_{\mathrm{I}}(u;\gamma)=\Ext_{\mathrm{II}}(u;\gamma^\sigma)\qquad\text{for all }\gamma\in\mathcal A_C.
\]
\begin{Hypotheses}\label{def:axioms}
We say that $(C,\mathcal A_C,\Ext_{\mathrm{I}},\Ext_{\mathrm{II}})$ satisfies the a hypotheses if:
\begin{enumerate}[label=\textup{(H\arabic*)}, leftmargin=2.8em]
\item it is Teichm\"uller regular in the sense of Definition~\ref{def:TeichReg};
\item it is degeneration--detecting in the sense of Definition~\ref{def:DegDetect};
\item it is pushable in the sense of Definition~\ref{def:Pushable}.
\end{enumerate}
\end{Hypotheses}

We can now state the first  result: under the Hypotheses \ref{def:axioms}, the function $H$ attains the value $0$.

\begin{theorem}[Existence of a reflexive point]\label{thm:main}
Let $C$ be a configuration datum and let $\Delta_C$ be the associated restricted character domain, with scale--fixed slice
$\Delta_C^{\mathrm{sc}}$.
Fix an admissible curve set $\mathcal A_C$ and continuous extremal--length assignments $\Ext_{\mathrm{I}},\Ext_{\mathrm{II}}$,
and define $H$ by \eqref{eq:height-def}.
If the  hypotheses~\ref{def:axioms} hold, then $\Delta_C^{\mathrm{sc}}$ contains a reflexive point.
\end{theorem}

The preceding theorem is formulated purely in terms of the functions $\Ext_{\mathrm{I}},\Ext_{\mathrm{II}}$ on
$\Delta_C^{\mathrm{sc}}$.
In geometric applications one is given, in addition, a structured pair of separate realization families
\[
u\longmapsto (X^{\mathrm{I}}(u),\omega_{\mathrm{I}}(u)),\qquad
u\longmapsto (X^{\mathrm{II}}(u),\omega_{\mathrm{II}}(u)),
\]
realizing the two characters separately and inducing the extremal--length assignments by
$\Ext_{\star}(u;\gamma)=\Ext_{X^{\star}(u)}(\gamma)$.
The next statement gives the geometric consequence one ultimately needs: at a reflexive parameter, the two separate
realizations come from a single punctured Riemann surface, so that both characters are realized simultaneously by two
meromorphic $1$--forms on that surface.

\begin{theorem}[Two meromorphic $1$--forms on one punctured surface]\label{thm:twoforms}
Assume the hypotheses of Theorem~\ref{thm:main}.
Assume moreover that the extremal--length assignments $\Ext_{\mathrm{I}},\Ext_{\mathrm{II}}$ are induced by continuous maps
\[
u\longmapsto X^{\mathrm{I}}(u),\ X^{\mathrm{II}}(u)\in\mathcal T_{g,n},
\]
and that for each $u=(\chi_{\mathrm{I}},\chi_{\mathrm{II}})\in\Delta_C^{\mathrm{sc}}$ there exist meromorphic $1$--forms
\[
\omega_{\mathrm{I}}(u)\in\Omega(X^{\mathrm{I}}(u)),\qquad
\omega_{\mathrm{II}}(u)\in\Omega(X^{\mathrm{II}}(u)),
\]
holomorphic on the punctured surface, with
\[
\chi_{\omega_{\mathrm{I}}(u)}=\chi_{\mathrm{I}},\qquad
\chi_{\omega_{\mathrm{II}}(u)}=\chi_{\mathrm{II}}.
\]
Let $u=(\chi_{\mathrm{I}},\chi_{\mathrm{II}})\in\Delta_C^{\mathrm{sc}}$ be a reflexive point.

Assume further that the finite curve system $\mathcal A_C$ is Teichm\"uller separating for the family image, in the sense that
whenever two marked punctured surfaces $Y,Z$ in this image satisfy
\[
\Ext_Y(\gamma)=\Ext_Z(\gamma^\sigma)\qquad\text{for all }\gamma\in\mathcal A_C,
\]
then there exists a puncture--preserving biholomorphism
\[
\Phi:Y\longrightarrow Z
\]
compatible with the markings, and then there exist a single marked punctured Riemann surface $X\in\Mgn$ and two meromorphic $1$--forms
$\eta_{\mathrm{I}},\eta_{\mathrm{II}}\in\Omega(X)$, holomorphic on the punctured surface, such that
\[
\chi_{\eta_{\mathrm{I}}}=\chi_{\mathrm{I}},\qquad \chi_{\eta_{\mathrm{II}}}=\chi_{\mathrm{II}}.
\]
\end{theorem}

\section{Configuration space of restricted characters}\label{sec:config_space_new}
Write $\Gamma=\Gamma_{g,n}$.  Throughout this section we fix a configuration datum
\[
C=(E,\iota,R,\tau,\sigma)
\]
in the sense of Definition~\ref{def:config}.  The goal is twofold:
\begin{itemize}[leftmargin=2.2em]
\item define the restricted character domain $\Delta_C$ and its scale--fixed slice $\Delta_C^{\mathrm{sc}}$;
\item give the auxiliary conformal data (admissible curves, extremal--length assignments appearing in Hypotheses~\ref{def:axioms}).
\end{itemize}

\subsection{The restricted character domain $\Delta_C$}\label{subsec:restricted-domains}
Given a character $\chi\in\Hom(\Gamma,\C)$, we encode its values on the distinguished classes $\iota(e)\in\Gamma$
by the coordinate vector
\[
p_\chi\ :=\ \bigl(\chi(\iota(e))\bigr)_{e\in E}\ \in\ \C^E .
\]
Since $\iota(E)$ generates $\Gamma$ (Definition~\ref{def:config}), the vector $p_\chi$ determines $\chi$
subject only to the relations encoded by $R=\ker(\iota_*:\Z^E\to\Gamma)$.
We write
\[
\mathcal R_h:=\R_{>0},\qquad \mathcal R_v:=i\R_{>0},\qquad
\overline{\mathcal R_h}=\mathcal R_h,\qquad \overline{\mathcal R_v}=-i\R_{>0}.
\]
Thus $\mathcal R_h$ (resp.\ $\mathcal R_v$) is the positive real (resp.\ positive imaginary) ray, and complex conjugation
fixes horizontals while reversing the orientation of verticals.
\begin{definition}[$C$--admissible pairs]\label{def:Cadm}
A pair of characters $(\chi_{\mathrm{I}},\chi_{\mathrm{II}})\in\Hom(\Gamma,\C)\times\Hom(\Gamma,\C)$ is
\emph{$C$--admissible} if there exist unit complex numbers $\zeta,\kappa\in S^1$ such that:
\begin{enumerate}[label=\textup{(C\arabic*)}, leftmargin=2.8em]
\item For every $e\in E$,
\[
\zeta^{-1}\chi_{\mathrm{I}}(\iota(e))\in \mathcal R_{\tau(e)},
\qquad
\zeta\,\kappa^{-1}\chi_{\mathrm{II}}(\iota(\sigma(e)))\in \overline{\mathcal R_{\tau(e)}} .
\]
\item For every relation $(r_e)_{e\in E}\in R$,
\[
\sum_{e\in E} r_e\,\chi_{\mathrm{I}}(\iota(e))\ =\ 0,
\qquad
\sum_{e\in E} r_e\,\chi_{\mathrm{II}}(\iota(\sigma(e)))\ =\ 0 .
\]
\item For every $e\in E$,
\[
\chi_{\mathrm{II}}(\iota(\sigma(e)))\ =\ \kappa\,\overline{\chi_{\mathrm{I}}(\iota(e))}.
\]
\end{enumerate}
\end{definition}

It is helpful to keep the following dictionary in mind.
Condition \textup{(C1)} enforces the horizontal/vertical combinatorics prescribed by $\tau$ (up to a global rotation $\zeta$),
\textup{(C2)} encodes the polygonal closing equations coming from $R$ (the ``period closure'' on each side),
and \textup{(C3)} expresses that the $\mathrm{II}$--periods are obtained from the $\mathrm{I}$--periods by conjugation, up to the
unit factor $\kappa$ (choice of reflection axis).

\begin{definition}[The domain $\Delta_C$]\label{def:DeltaC}
The \emph{restricted character domain} associated to $C$ is
\[
\Delta_C\ :=\ \Bigl\{(\chi_{\mathrm{I}},\chi_{\mathrm{II}})\in\Hom(\Gamma,\C)^2:\ (\chi_{\mathrm{I}},\chi_{\mathrm{II}})
\text{ is $C$--admissible}\Bigr\}.
\]
\end{definition}
\medskip

Choose a distinguished edge $e_0\in E$ with $\tau(e_0)=h$.
For any $(\chi_{\mathrm{I}},\chi_{\mathrm{II}})\in\Delta_C$, the ray conditions imply that
$\chi_{\mathrm{I}}(\iota(e_0))\neq 0$ lies on a ray.  Using the parameters $(\zeta,\kappa)$ from
Definition~\ref{def:Cadm}, we eliminate the inessential choices of rotation and reflection by imposing
\[
\chi_{\mathrm{I}}(\iota(e_0))\in\R_{>0},
\qquad
\chi_{\mathrm{II}}(\iota(\sigma(e)))=\overline{\chi_{\mathrm{I}}(\iota(e))}\ \ \forall e\in E .
\]
This leads to the normalized subdomain
\[
\Delta_C^{\mathrm{norm}}
:=\Bigl\{(\chi_{\mathrm{I}},\chi_{\mathrm{II}})\in \Delta_C:
\chi_{\mathrm{I}}(\iota(e_0))\in\R_{>0}
\ \text{and}\ 
\chi_{\mathrm{II}}(\iota(\sigma(e)))=\overline{\chi_{\mathrm{I}}(\iota(e))}\ \ \forall e\in E
\Bigr\}.
\]
On $\Delta_C^{\mathrm{norm}}$ the directional constraints take the simpler form
\[
\chi_{\mathrm{I}}(\iota(e))\in \mathcal R_{\tau(e)},
\qquad
\chi_{\mathrm{II}}(\iota(\sigma(e)))\in \overline{\mathcal R_{\tau(e)}},
\qquad \forall e\in E.
\]

\begin{lemma}\label{lem:opencone}
The set $\Delta_C^{\mathrm{norm}}$ is an open cone in a finite--dimensional real vector space.
\end{lemma}

\begin{proof}
Write $x_e:=\chi_{\mathrm{I}}(\iota(e))$ for $e\in E$.  The ray constraints force
\[
x_e\in\R_{>0}\ \text{ if }\tau(e)=h,\qquad x_e\in i\R_{>0}\ \text{ if }\tau(e)=v.
\]
Hence $(x_e)_{e\in E}$ lies in an open orthant in the real subspace
\[
V_C\ :=\ \Bigl\{(z_e)_{e\in E}\in\C^E:
z_e\in\R \text{ if }\tau(e)=h,\ z_e\in i\R \text{ if }\tau(e)=v,\
\sum_{e\in E} r_e z_e=0\ \ \forall (r_e)\in R\Bigr\}.
\]
The normalization $\chi_{\mathrm{II}}(\iota(\sigma(e)))=\overline{x_e}$ determines $\chi_{\mathrm{II}}$ uniquely from $(x_e)_{e\in E}$,
so $\Delta_C^{\mathrm{norm}}$ identifies with an open orthant in $V_C$.  Positive real scaling of all $x_e$ preserves the defining
conditions, giving the cone structure.
\end{proof}

The only remaining freedom on $\Delta_C^{\mathrm{norm}}$ is simultaneous \emph{positive real scaling} of both characters.
We therefore pass to the scale--fixed slice.

\begin{definition}\label{def:DeltaCsc}
The scale--fixed slice of $\Delta_C$ is the codimension--one submanifold
\[
\Delta_C^{\mathrm{sc}}
\ :=\
\Bigl\{(\chi_{\mathrm{I}},\chi_{\mathrm{II}})\in \Delta_C^{\mathrm{norm}}:\ \chi_{\mathrm{I}}(\iota(e_0))=1\Bigr\}.
\]
\end{definition}
\begin{lemma}\label{lem:DeltaCsc_manifold}
The scale--fixed slice $\Delta_C^{\mathrm{sc}}$ is a smooth codimension--one submanifold of
$\Delta_C^{\mathrm{norm}}$. In particular, $\Delta_C^{\mathrm{sc}}$ is a (finite--dimensional) $C^\infty$ manifold.
\end{lemma}

\begin{proof}
By Lemma~\ref{lem:opencone}, $\Delta_C^{\mathrm{norm}}$ is an open subset of the finite--dimensional real vector space
$V_C$. Write $x_e:=\chi_{\mathrm{I}}(\iota(e))$ for $e\in E$. Since $\tau(e_0)=h$, the coordinate $x_{e_0}$ is real on
$V_C$, and hence defines a smooth (indeed linear) map
\[
F:\Delta_C^{\mathrm{norm}}\longrightarrow \R,\qquad
F\bigl((\chi_{\mathrm{I}},\chi_{\mathrm{II}})\bigr):=\chi_{\mathrm{I}}(\iota(e_0))=x_{e_0}.
\]
The differential $dF_u:T_u\Delta_C^{\mathrm{norm}}\to\R$ is the restriction of the nonzero linear functional
$V_C\to\R$, $(z_e)_{e\in E}\mapsto z_{e_0}$, so $dF_u$ is surjective for every $u\in\Delta_C^{\mathrm{norm}}$.
Thus $1\in\R$ is a regular value of $F$. By the regular value theorem,
\[
\Delta_C^{\mathrm{sc}}
=\bigl\{u\in\Delta_C^{\mathrm{norm}}:\chi_{\mathrm{I}}(\iota(e_0))=1\bigr\}
=F^{-1}(1)
\]
is a smooth submanifold of $\Delta_C^{\mathrm{norm}}$ of codimension one.
\end{proof}

\subsection{Admissible curves and extremal--length assignments}

To compare the two sides, we fix a finite set of test curve families on $S_{g,n}$, together with a pairing
between the $\mathrm{I}$-- and $\mathrm{II}$--families, and we prescribe (or construct) their extremal lengths as functions of
the parameter $u\in\Delta_C^{\mathrm{sc}}$.

\noindent
Fix an admissible curve set $\mathcal A_C$ for $C$ in the sense of Definition~\ref{def:AC}, and write
$\gamma^\sigma:=\sigma_\ast(\gamma)$ for the paired curve family.

\begin{definition}[Extremal--length assignment]\label{def:ExtAssign}
An \emph{extremal--length assignment} for $C$ consists of two maps
\[
\Ext_{\mathrm{I}}:\Delta_C^{\mathrm{sc}}\times \mathcal A_C\longrightarrow \R_{>0},
\qquad
\Ext_{\mathrm{II}}:\Delta_C^{\mathrm{sc}}\times \mathcal A_C\longrightarrow \R_{>0},
\]
such that for each fixed $\gamma\in\mathcal A_C$ the functions
\[
u\longmapsto \Ext_{\star}(u;\gamma),\qquad \star\in\{\mathrm{I},\mathrm{II}\},
\]
are continuous on $\Delta_C^{\mathrm{sc}}$.
\end{definition}

Given $\mathcal A_C$ and an extremal--length assignment, we define the mismatch functions and the height exactly as in
Section~\ref{sec:terminology}.  Concretely, for $\gamma\in\mathcal A_C$ set
\[
m_\gamma(u)\ :=\ \log\frac{\Ext_{\mathrm{I}}(u;\gamma)}{\Ext_{\mathrm{II}}(u;\gamma^\sigma)},
\qquad u\in\Delta_C^{\mathrm{sc}},
\]
and define
\[
H(u)\ :=\ \sum_{\gamma\in\mathcal A_C} m_\gamma(u)^2\ \ge 0 .
\]

\subsection{Analytic hypotheses on extremal length data}\label{subsec:C_axioms}

The abstract analytic requirements used in the main results are given in Hypotheses~\ref{def:axioms}.
For the arguments in Section~\ref{sec:proof_main} it is convenient to give three concrete properties of the maps
$u\mapsto \Ext_{\mathrm{I}}(u;\gamma)$ and $u\mapsto \Ext_{\mathrm{II}}(u;\gamma)$.

\begin{definition}[Teichm\"uller regularity]\label{def:TeichReg}
We say that $C$ is \emph{Teichm\"uller regular} with respect to the chosen curve set and extremal--length assignment if:
\begin{enumerate}[label=\textup{(R\arabic*)}, leftmargin=2.8em]
\item for each $\gamma\in\mathcal A_C$, the functions $u\longmapsto \Ext_{\star}(u;\gamma)$ are $C^1$ on $\Delta_C^{\mathrm{sc}}$ for
$\star\in\{\mathrm{I},\mathrm{II}\}$;
\item the log--extremal--length maps
\[
\mathbf{E}_{\star}:\Delta_C^{\mathrm{sc}}\longrightarrow \R^{\mathcal A_C},\qquad
\mathbf{E}_{\star}(u):=\bigl(\log \Ext_{\star}(u;\gamma)\bigr)_{\gamma\in\mathcal A_C},
\]
are local immersions for $\star\in\{\mathrm{I},\mathrm{II}\}$.
\end{enumerate}
\end{definition}

For a fixed curve family $\gamma$, the function $X\mapsto \Ext_X(\gamma)$ on Teichm\"uller space is real--analytic and admits
Gardiner’s variational formula; see \cite{Kerckhoff,GardinerQD,HubbardMasur} and \cite[Chapters~6--7]{GardinerLakic}.
Thus \textup{(R1)} is automatic once the parameter maps into Teichm\"uller space are $C^1$.
Moreover, \textup{(R2)} is equivalent to the statement that at each $u$ the covectors
$\{d\log\Ext_{\star}(u;\gamma)\}_{\gamma\in\mathcal A_C}$ span $T_u^\ast(\Delta_C^{\mathrm{sc}})$.

\begin{definition}[Degeneration detection]\label{def:DegDetect}
We say that $(C,\mathcal A_C)$ is \emph{degeneration--detecting} if for every sequence
$u_k\in\Delta_C^{\mathrm{sc}}$ leaving every compact subset of $\Delta_C^{\mathrm{sc}}$, there exists a curve
$\gamma\in\mathcal A_C$ such that
\[
|m_\gamma(u_k)|\longrightarrow\infty .
\]
\end{definition}

\begin{definition}[Pushability]\label{def:Pushable}
We say that $(C,\mathcal A_C)$ is \emph{pushable} if for each $\gamma\in\mathcal A_C$ there exist:
\begin{itemize}[leftmargin=2.2em]
\item a $C^1$ vector field $V_\gamma$ on $\Delta_C^{\mathrm{sc}}$, and
\item a finite incidence set $I(\gamma)\subset \mathcal A_C$ with $\gamma\in I(\gamma)$,
\end{itemize}
such that for every $u\in\Delta_C^{\mathrm{sc}}$ the following hold.
\begin{enumerate}[label=\textup{(P\arabic*)}, leftmargin=2.8em]
\item 
If $m_\gamma(u)\neq 0$, then
\(
\mathrm{sign}\bigl(m_\gamma(u)\bigr)\cdot (dm_\gamma)_u(V_\gamma) <0.
\)
\item 
If $\delta\notin I(\gamma)$ then $(dm_\delta)_u(V_\gamma)=0$. For $\delta\in I(\gamma)$ the derivatives
$(dm_\delta)_u(V_\gamma)$ are uniformly bounded in $u$.
\item 
If $m_\gamma(u)\neq 0$, then
\(
|m_\gamma(u)|\cdot\bigl|(dm_\gamma)_u(V_\gamma)\bigr|
\;>\;
\sum_{\delta\in I(\gamma)\setminus\{\gamma\}}
|m_\delta(u)|\cdot\bigl|(dm_\delta)_u(V_\gamma)\bigr|.
\)
\end{enumerate}
\end{definition}

\section{Proof of main theorems}\label{sec:proof_main}

Throughout this section, fix a configuration datum $C$ together with an admissible curve set $\mathcal A_C$ and an
extremal--length assignment $\Ext_{\mathrm{I}},\Ext_{\mathrm{II}}$ as in Section~\ref{sec:config_space_new}.
We prove Theorem~\ref{thm:main} and Theorem~\ref{thm:twoforms} stated in Section~\ref{sec:terminology}.

\medskip

We begin by writing a standard result for nonnegative functions.

\begin{proposition}\label{prop:proper-min}
Let $\Omega$ be a smooth manifold and let $f:\Omega\longrightarrow\R_{\ge 0}$ be a continuous function.
Assume that $f$ is proper on $\Omega$, and that every local minimizer of $f$ lies in the zero--set $f^{-1}(0)$.
Then $f$ attains its minimum on $\Omega$, and this minimum equals $0$.
\end{proposition}

\begin{proof}
Set $m:=\inf_{\Omega} f\in\R_{\ge 0}$ and choose a minimizing sequence $(x_k)_{k\ge 1}\subset \Omega$ such that
$f(x_k)\longrightarrow m$.

Fix any $A>m$. Since $f(x_k)\longrightarrow m$, there exists $k_0$ such that $f(x_k)\le A$ for all $k\ge k_0$, hence
$\{x_k:k\ge k_0\}\subset K_A:=f^{-1}([0,A])$.
Properness of $f$ implies that $K_A$ is compact. Therefore $(x_k)_{k\ge k_0}$ has a convergent subsequence;
relabel so that $x_k\longrightarrow x_\infty\in K_A\subset\Omega$.

By continuity of $f$ we obtain
\[
f(x_\infty)=\lim_{k\longrightarrow\infty} f(x_k)=m,
\]
so $x_\infty$ is a global minimizer of $f$. In particular, $x_\infty$ is a local minimizer, hence by hypothesis
$x_\infty\in f^{-1}(0)$. Thus $m=f(x_\infty)=0$, and $f$ attains its minimum value $0$ on $\Omega$.
\end{proof}

We now verify properness of the height function $H$ under degeneration detection.

\begin{proposition}[Properness criterion]\label{prop:PropernessCriterion}
If $(C,\mathcal A_C)$ is degeneration--detecting in the sense of Definition~\ref{def:DegDetect}, then the
height function $H$ is proper on $\Delta_C^{\mathrm{sc}}$.
\end{proposition}

\begin{proof}
Fix $A\ge 0$ and consider the sublevel set
\[
K_A:=H^{-1}([0,A])=\{u\in \Delta_C^{\mathrm{sc}}: H(u)\le A\}.
\]
By Lemma~\ref{lem:DeltaCsc_manifold}, the space $\Delta_C^{\mathrm{sc}}$ is a finite--dimensional smooth manifold, hence
metrizable; therefore it suffices to show that every sequence in $K_A$ admits a convergent subsequence with limit still in
$K_A$.

Let $(u_k)$ be a sequence in $K_A$. If $(u_k)$ had no convergent subsequence, then it would leave every compact subset of
$\Delta_C^{\mathrm{sc}}$: indeed, if some compact $K\subset\Delta_C^{\mathrm{sc}}$ contained infinitely many terms of the
sequence, then those terms would admit a convergent subsequence by compactness of $K$.

Hence, assuming for contradiction that $(u_k)$ has no convergent subsequence, degeneration detection provides
$\gamma\in\mathcal A_C$ such that $|m_\gamma(u_k)|\longrightarrow\infty$. But by definition of $H$,
\[
H(u_k)=\sum_{\delta\in\mathcal A_C} m_\delta(u_k)^2 \;\ge\; m_\gamma(u_k)^2,
\]
so $H(u_k)\longrightarrow\infty$, contradicting the uniform bound $H(u_k)\le A$.
Therefore $(u_k)$ has a convergent subsequence $u_{k_j}\longrightarrow u_\infty$ in $\Delta_C^{\mathrm{sc}}$.

Finally, $\mathcal A_C$ is finite and the extremal--length assignments are continuous by
Definition~\ref{def:ExtAssign}, so each mismatch $m_\delta$ (hence $m_\delta^2$) is continuous on $\Delta_C^{\mathrm{sc}}$.
Thus $H$ is continuous, and
\[
H(u_\infty)=\lim_{j\longrightarrow\infty}H(u_{k_j})\le A,
\]
i.e.\ $u_\infty\in K_A$.
Thus $K_A$ is compact, and $H$ is proper.
\end{proof}

Next we verify the `no positive local minima'' property for $H$ under Teichm\"uller regularity and pushability.

\begin{proposition}\label{prop:NoNonreflexiveMin}
Assume $C$ is Teichm\"uller regular (Definition~\ref{def:TeichReg}) and $(C,\mathcal A_C)$ is pushable
(Definition~\ref{def:Pushable}). Then every local minimizer of $H$ is reflexive.
\end{proposition}

\begin{proof}
Recall that
\[
H(u)=\sum_{\delta\in\mathcal A_C} m_\delta(u)^2.
\]
By Teichm\"uller regularity, each $m_\delta$ is $C^1$, hence $H$ is $C^1$ on $\Delta_C^{\mathrm{sc}}$.

Let $u\in\Delta_C^{\mathrm{sc}}$ be non--reflexive. Then $H(u)>0$, so there exists $\gamma\in\mathcal A_C$ with
$m_\gamma(u)\neq 0$. Let $V_\gamma$ denote the push vector field provided by pushability.
Differentiating $H$ in the direction $V_\gamma$ gives
\begin{equation}\label{eq:dH-along-V}
(dH)_u(V_\gamma)
=2\sum_{\delta\in\mathcal A_C} m_\delta(u)\cdot (dm_\delta)_u(V_\gamma).
\end{equation}
By \textup{(P2)}, $(dm_\delta)_u(V_\gamma)=0$ for $\delta\notin I(\gamma)$, hence the sum in
\eqref{eq:dH-along-V} may be restricted to $\delta\in I(\gamma)$.

For the $\delta=\gamma$ term, \textup{(P1)} gives
\[
\mathrm{sign}(m_\gamma(u))\cdot (dm_\gamma)_u(V_\gamma)<0,
\]
so $(dm_\gamma)_u(V_\gamma)$ has the opposite sign to $m_\gamma(u)$ and therefore
\[
m_\gamma(u)\,(dm_\gamma)_u(V_\gamma)=-\,|m_\gamma(u)|\,\bigl|(dm_\gamma)_u(V_\gamma)\bigr|.
\]
For each $\delta\in I(\gamma)\setminus\{\gamma\}$ we simply use the estimate
\[
m_\delta(u)\,(dm_\delta)_u(V_\gamma)
\le |m_\delta(u)|\,\bigl|(dm_\delta)_u(V_\gamma)\bigr|.
\]
Substituting these inequalities into \eqref{eq:dH-along-V} gives
\[
(dH)_u(V_\gamma)
\le
-2\,|m_\gamma(u)|\,\bigl|(dm_\gamma)_u(V_\gamma)\bigr|
+2\sum_{\delta\in I(\gamma)\setminus\{\gamma\}}
|m_\delta(u)|\,\bigl|(dm_\delta)_u(V_\gamma)\bigr|.
\]
By \textup{(P3)} the right--hand side is strictly negative, hence
\begin{equation}\label{eq:dH-negative}
(dH)_u(V_\gamma)<0.
\end{equation}

Choose a smooth curve $\varphi:(-\varepsilon,\varepsilon)\to \Delta_C^{\mathrm{sc}}$ such that
$\varphi(0)=u$ and $\varphi'(0)=V_\gamma(u)$ (for instance, take $\varphi$ to be a straight line in a local coordinate chart).
The chain rule gives
\[
\left.\frac{d}{dt}\right|_{t=0} H(\varphi(t))=(dH)_u(V_\gamma)<0.
\]
Therefore $H(\varphi(t))<H(u)$ for all sufficiently small $t>0$, which contradicts the assumption that $u$ is a
local minimizer of $H$. Hence every local minimizer of $H$ must be reflexive.
\end{proof}


\subsection{Existence of reflexive points: proof of Theorem~\ref{thm:main}}\label{sec:main_theorem}

\begin{proof}[Proof of Theorem~\ref{thm:main}]
Assume $C$ satisfies the Hypothesis~\ref{def:axioms}. In particular, $(C,\mathcal A_C)$ is
degeneration--detecting, Teichm\"uller regular, and pushable.

By Proposition~\ref{prop:PropernessCriterion}, the height function $H$ is proper on $\Delta_C^{\mathrm{sc}}$.
By Proposition~\ref{prop:NoNonreflexiveMin}, every local minimizer of $H$ is reflexive; equivalently, every local minimizer
lies in $H^{-1}(0)$.

Apply Proposition~\ref{prop:proper-min} to $f=H$ on $\Omega=\Delta_C^{\mathrm{sc}}$.
We conclude that $H$ attains its minimum on $\Delta_C^{\mathrm{sc}}$, and that this minimum equals $0$.
Thus there exists $u\in\Delta_C^{\mathrm{sc}}$ with $H(u)=0$, i.e.\ a reflexive point.
\end{proof}

\subsection{Two compatible characters on one surface: proof of Theorem~\ref{thm:twoforms}}\label{subsec:twoforms}

For $u=(\chi_{\mathrm{I}},\chi_{\mathrm{II}})\in\Delta_C$ we assume, as in the statement of
Theorem~\ref{thm:twoforms}, that $u$ comes with a \emph{realization} by two marked punctured Riemann surfaces
$X^{\mathrm{I}}(u),X^{\mathrm{II}}(u)$ and meromorphic differentials
\[
\omega_{\mathrm{I}}(u)\in\Omega\bigl(X^{\mathrm{I}}(u)\bigr),\qquad
\omega_{\mathrm{II}}(u)\in\Omega\bigl(X^{\mathrm{II}}(u)\bigr),
\]
whose period characters are $\chi_{\mathrm{I}},\chi_{\mathrm{II}}$, respectively.

\begin{corollary}\label{cor:twoforms}
Assume that $\mathcal A_C$ is Teichm\"uller separating for the family
\[
u\longmapsto \bigl(X^{\mathrm{I}}(u),X^{\mathrm{II}}(u)\bigr)
\qquad (u\in\Delta_C)
\]
in the sense of Theorem~\ref{thm:twoforms}.
If $u\in\Delta_C^{\mathrm{sc}}$ is reflexive, then $X^{\mathrm{I}}(u)$ and $X^{\mathrm{II}}(u)$ are conformally equivalent.
\end{corollary}

\begin{proof}
Let $u\in\Delta_C^{\mathrm{sc}}$ be reflexive. By definition of reflexivity, $H(u)=0$, hence
\[
m_\gamma(u)=0\qquad\text{for every }\gamma\in\mathcal A_C.
\]
Equivalently, for every $\gamma\in\mathcal A_C$ we have the extremal--length matching
\begin{equation}\label{eq:EL-match-cor}
\Ext_{\mathrm{I}}(u;\gamma)=\Ext_{\mathrm{II}}(u;\gamma^\sigma).
\end{equation}
By the Teichm\"uller separating hypothesis, the collection of equalities \eqref{eq:EL-match-cor} forces the existence of a
puncture--preserving biholomorphism
\[
\Phi:\,X^{\mathrm{I}}(u)\longrightarrow X^{\mathrm{II}}(u)
\]
that is compatible with the chosen markings in the $\sigma_\ast$--convention.
\end{proof}

\begin{proof}[Proof of Theorem~\ref{thm:twoforms}]
Let $u=(\chi_{\mathrm{I}},\chi_{\mathrm{II}})\in\Delta_C^{\mathrm{sc}}$ be reflexive and choose the associated realizations
\[
\bigl(X^{\mathrm{I}}(u),\omega_{\mathrm{I}}(u)\bigr),\qquad
\bigl(X^{\mathrm{II}}(u),\omega_{\mathrm{II}}(u)\bigr)
\]
as in the statement. Set $X^{\mathrm{I}}:=X^{\mathrm{I}}(u)$ and $X^{\mathrm{II}}:=X^{\mathrm{II}}(u)$ for brevity.

By Corollary~\ref{cor:twoforms},  there exists a puncture--preserving biholomorphism
\[
\Phi:\,X^{\mathrm{I}}\longrightarrow X^{\mathrm{II}}
\]
compatible with the markings in the $\sigma_\ast$--convention.

Define a single punctured surface and two meromorphic $1$--forms on it by
\[
X:=X^{\mathrm{I}},\qquad
\eta_{\mathrm{I}}:=\omega_{\mathrm{I}}(u)\in\Omega(X),\qquad
\eta_{\mathrm{II}}:=\Phi^\ast\omega_{\mathrm{II}}(u)\in\Omega(X).
\]
Then $\eta_{\mathrm{I}}$ and $\eta_{\mathrm{II}}$ are meromorphic on $X$ and holomorphic on the punctured surface.

\medskip
\noindent\textit{Digression: Marking compatibility.}
Let $\mu_{\mathrm{I}}:S_{g,n}\cong X^{\mathrm{I}}\setminus P^{\mathrm{I}}$ and
$\mu_{\mathrm{II}}:S_{g,n}\cong X^{\mathrm{II}}\setminus P^{\mathrm{II}}$ denote the chosen markings.
In particular, these markings identify $\Gamma_{g,n}=H_1(S_{g,n};\Z)$ with the homology of the punctured surfaces.
Compatibility of $\Phi$ with the markings (in the $\sigma_\ast$--convention) means that the induced map on homology makes the
diagram
\begin{center}
\begin{tikzpicture}[>=Stealth]
\node (G) at (0,0) {$\Gamma_{g,n}$};
\node (HI) at (4,1) {$H_1(X^{\mathrm{I}}\setminus P^{\mathrm{I}};\Z)$};
\node (HII) at (4,-1) {$H_1(X^{\mathrm{II}}\setminus P^{\mathrm{II}};\Z)$};
\draw[->] (G) -- node[above left] {$\ (\mu_{\mathrm{I}})_\ast$} (HI);
\draw[->] (G) -- node[below left] {$\ (\mu_{\mathrm{II}})_\ast$} (HII);
\draw[->] (HI) -- node[right] {$\Phi_\ast$} (HII);
\end{tikzpicture}
\end{center}
commute (the dependence on $\sigma_\ast$ is encoded in how the markings are chosen in the hypothesis of
Theorem~\ref{thm:twoforms}).

It remains to identify the period characters.
Fix $\gamma\in\Gamma_{g,n}$ and represent $(\mu_{\mathrm{I}})_\ast(\gamma)$ by a $1$--cycle $c$ on the punctured surface
$X=X^{\mathrm{I}}$.
By definition of pullback,
\[
\int_c \eta_{\mathrm{II}}
=
\int_c \Phi^\ast\omega_{\mathrm{II}}(u)
=
\int_{\Phi_\ast c}\omega_{\mathrm{II}}(u).
\]
By the commutative diagram above, the cycle $\Phi_\ast c$ represents $(\mu_{\mathrm{II}})_\ast(\gamma)$ on
$X^{\mathrm{II}}$, hence
\[
\int_{\Phi_\ast c}\omega_{\mathrm{II}}(u)=\chi_{\mathrm{II}}(\gamma).
\]
Therefore $\chi_{\eta_{\mathrm{II}}}=\chi_{\mathrm{II}}$.
Similarly, since $\eta_{\mathrm{I}}=\omega_{\mathrm{I}}(u)$ on $X^{\mathrm{I}}$, we have $\chi_{\eta_{\mathrm{I}}}=\chi_{\mathrm{I}}$.
This proves Theorem~\ref{thm:twoforms}.
\end{proof}

The argument above separates cleanly into two inputs.
First, reflexivity provides the extremal--length equalities \eqref{eq:EL-match-cor} along the admissible set $\mathcal A_C$.
Second, the Teichm\"uller separating property asserts that these equalities already determine the conformal structure of the
marked punctured surface (up to the $\sigma_\ast$ marking convention), and hence produce the biholomorphism $\Phi$.
In concrete minimal--surface applications, low--complexity configurations can fail precisely because one cannot choose
$\mathcal A_C$ to satisfy both roles at once: if $\mathcal A_C$ is too small, it may not detect degenerations or may fail to
separate Teichm\"uller space, and then the minimization and/or identification steps break down.

\section{Model cases for the synchronization method}\label{sec:examples}

Hypotheses~\ref{def:axioms} are formulated to distill an analytic synchronization mechanism
that is already present in the classical minimal--surface constructions of Weber and Wolf.
There, one starts with two {a priori} different families of conformal structures arising from paired
planar domains (orthodisks) with prescribed combinatorics, compares extremal lengths of distinguished
(edge--parallel) curve families on the two sides, and minimizes a {height} function given by a sum of
squared log--ratios in order to force equality of the conformal structures.  In the setting of
\cite{WeberWolfGAFA,WeberWolfAnnals}, Teichm\"uller regularity, degeneration detection, and pushability are
verified by explicit extremal--length estimates for these edge--parallel cycles, and the resulting reflexive
parameter is exactly the point at which the two sides synchronize, yielding the desired Weierstrass data for
the minimal surface.  In this sense, the results of
\cite{WeberWolfGAFA,WeberWolfAnnals,BardhanBiswasFujimoriKumarAngel,BardhanBiswasKumarMaxfaces}
may be viewed as direct example of the framework developed here.

\medskip
The examples in this section are organized slightly differently from the formal construction
in Definition~\ref{def:config}.  Rather than beginning by specifying a discrete configuration datum
\(C=(E,\iota,\tau,R,\dots)\) and then extracting the corresponding \(C\)-compatible restricted characters, we
start from concrete surface families for which the relevant compatibility relation between the two restricted
characters is visible by inspection.  The point is not to change the notion of compatibility, but to
emphasize that the height--minimization/synchronization method applies whenever one can produce
two restricted characters related by a fixed compatibility constraint on a scale--fixed slice---even if that
constraint is realized by different choices of discrete data \(C\) in the background.

\medskip
We give two complementary examples.  In the first, we work with a completely explicit family and verify
Hypotheses~\ref{def:axioms} by direct extremal--length computations; Theorem~\ref{thm:main} then produces a
\emph{reflexive} parameter, and Corollary~\ref{cor:twoforms} yields a simultaneous realization on a single
punctured surface.  In the second, we show that in a natural configuration the edge--parallel admissible
class is too small to produce Teichm\"uller--regular extremal--length coordinates, so
Hypothesis~\ref{def:axioms}(H1) fails and the main theorem cannot be applied.

\medskip

\subsection{A configuration where the theorem works}\label{subsec:good_example}

In the article, the two sides are organized starting from a discrete configuration datum \(C\),
which packages the combinatorics (edge set, incidence, involutions/permutations, gluing data, etc.) and
thereby produces a restricted character domain \(\Delta_C^{\mathrm{sc}}\) and a notion of \(C\)-compatible
pairs of restricted characters.
In the first example below we deliberately \emph{present the same structure from the opposite direction}:
we begin with a very concrete geometric family and write down directly the induced relation between the two
restricted characters (hence the two marked surfaces) that plays the role of “compatibility”.
Although this relation is expressed in a simpler form than the general definition, it is not a new
framework: it is precisely the compatibility induced by an appropriate choice of discrete datum \(C\) on the
relevant slice.  In other words, we are not changing the method; we are choosing a particularly transparent
model in which the induced \(C\)-compatibility can be read off immediately.

The point of this example is to make the abstract direction of the main theorem feel concrete.
Abstractly, one starts with a configuration datum \(C\), hence a restricted character domain
\(\Delta_C^{\mathrm{sc}}\), and then seeks a parameter \(u\in \Delta_C^{\mathrm{sc}}\) for which the two
associated marked surfaces \(\,X_{\mathrm{I}}(u)\) and \(X_{\mathrm{II}}(u)\,\) \emph{synchronize} (i.e.\ represent the same point of Teichm\"uller space).
In this example we keep the same logical order, but we describe the objects in a very explicit geometric model,
so the reader can see immediately what the restricted characters are and what synchronization means.   

\medskip

Fix a genus--two surface \(S_2\) and choose simple closed curves
\(\alpha_1,\beta_1,\alpha_2,\beta_2\in \Gamma:=\mathrm{H}_1(S_2\,;\,\Z)\)
with the usual symplectic intersection pattern.
Let
\[
J:\Gamma\longrightarrow \Gamma
\]
be the automorphism determined by
\(
J(\alpha_i)=\beta_i
\)
and
\(
J(\beta_i)=-\alpha_i
\)
for \(i=1,2\).
Intuitively, \(J\) exchanges the ``horizontal'' and ``vertical'' directions on each torus component.

\medskip

If \((X,\omega)\) is a Riemann surface with a holomorphic \(1\)--form and
\(f:S_2\to X\) is a marking, we denote the corresponding period character by
\[
\chi_{\omega}^{\,f}:\Gamma\longrightarrow \C,
\qquad
\chi_{\omega}^{\,f}(\gamma)=\int_{f_*(\gamma)}\omega.
\]
Precomposing the marking by \(J\) produces a second character
\[
\chi_{\omega}^{\,f\circ J}=\chi_{\omega}^{\,f}\circ J.
\]
Thus the basic ``compatibility'' built into this configuration is simply:
\[
\chi_{\mathrm{II}}=\chi_{\mathrm{I}}\circ J.
\]
Equivalently, the same underlying pair \((X,\omega)\) gives two \emph{marked} surfaces
\[
X_{\mathrm{I}}:=(X,f),
\qquad
X_{\mathrm{II}}:=(X,f\circ J),
\]
which need not represent the same point of Teichm\"uller space.
The synchronization problem for this pair is:

\begin{center}
\emph{When does there exist a conformal self--map of \(X\) whose action on \(\Gamma\) is \(J\)?}
\end{center}

\medskip

We now build an explicit family \(X(u)\) for which the associated restricted pair
\((\chi_{\mathrm{I}}(u),\chi_{\mathrm{II}}(u))\) is easy to write down, and then in \S5.1.2 we compute
extremal lengths on \(X(u)\) that detect synchronization.

\subsubsection{From the surface to a configuration space (a scale--fixed slice)}

Fix positive real numbers \(a,b,c\) and a slit length \(\ell\) with \(0<\ell<\min\{b,c\}\).
Let
\[
R_1=[0,a]\times[0,b],
\qquad
R_2=[0,a]\times[0,c],\]
with Euclidean coordinate \(z=x+iy\).
Form the rectangular tori
\[
T_1:=R_1/\{(0,y)\sim(a,y)\ \text{and}\ (x,0)\sim(x,b)\},
\;
T_2:=R_2/\{(0,y)\sim(a,y)\ \text{and}\ (x,0)\sim(x,c)\}.
\]
Let \(s_1\subset T_1\) and \(s_2\subset T_2\) be the images of the vertical segments
\[
\{a/2\}\times[0,\ell]\subset R_1,
\qquad
\{a/2\}\times[0,\ell]\subset R_2.
\]
Cut each \(T_j\) open along \(s_j\), and glue the two cut boundaries crosswise by translation.
The result is a connected translation surface \(X(a,b,c)\) of genus \(2\) carrying the holomorphic
\(1\)--form induced by \(dz\).

Denote by \(\alpha_1\) (resp.\ \(\alpha_2\)) the horizontal core curve on the \(T_1\) (resp.\ \(T_2\)) side,
and by \(\beta_1\) (resp.\ \(\beta_2\)) the vertical core curve on the \(T_1\) (resp.\ \(T_2\)) side.
Choose a marking \(f:S_2\to X(a,b,c)\) sending the fixed topological curves
\(\alpha_i,\beta_i\subset S_2\) to these geometric core curves; we keep the notation \(f\) implicit in what follows.

The period character of \(dz\) on \(X(a,b,c)\) is
\[
\chi_{dz}(\alpha_1)=\chi_{dz}(\alpha_2)=a,
\qquad
\chi_{dz}(\beta_1)=i\,b,
\qquad
\chi_{dz}(\beta_2)=i\,c.
\]
This is the restriction that defines the relevant slice of the character domain in this example: the two
``horizontal'' periods agree.

We view \((b,c)\) as a scale--fixed slice by setting \(a=1\).  For notational simplicity we suppress the
subscript \(C\) and write
\[
\Delta^{\mathrm{sc}}_{}
:=\{(b,c)\in\R_{>0}^2:\ \ell<\min\{b,c\}\},
\qquad
u=(b,c)\in \Delta^{\mathrm{sc}}_{},
\]
and we define
\[
X(u):=X(1,b,c),\qquad \omega(u):=dz\ \text{on}\ X(u),
\]
together with the two marked surfaces
\[
X_{\mathrm{I}}(u):=(X(u),f),
\qquad
X_{\mathrm{II}}(u):=(X(u),f\circ J).
\]
Finally, the associated \(C\)-compatible pair of characters is
\[
\chi_{\mathrm{I}}(u):=\chi_{\omega(u)}^{\,f},
\qquad
\chi_{\mathrm{II}}(u):=\chi_{\omega(u)}^{\,f\circ J}
=\chi_{\mathrm{I}}(u)\circ J.
\]
This is the concrete ``two surfaces from two restricted characters''  in the present example.


\subsubsection{Slice calculation: extremal lengths and the role of the slit length}

The surface \(X(1,b,c)\) depends on the slit length \(\ell\), and the horizontal foliation is indeed
cut by the transverse slit for heights \(0<y<\ell\).  In particular, we should not claim that the
entire \(T_1\) side forms a single horizontal Euclidean cylinder.
What is true (and what we use below) is that the horizontal and vertical directions on each torus side
still contain maximal Euclidean cylinders whose boundaries are saddle connections through the cone
points created by the slit endpoints; these maximal cylinders determine the corresponding extremal lengths.

\medskip
We recall a standard fact about extremal length on Euclidean cylinders, which we use repeatedly in the sequel:  
Let \(A\) be the Euclidean cylinder obtained from a rectangle of width \(w>0\) and height \(h>0\) by
identifying the vertical sides.  Let \(\gamma\) be the free homotopy class of the core curve (horizontal
circle).  Then
\[
\Ext_A(\gamma)=\frac{w}{h}.
\]

Now we compute the extremal lengths of the core curves \(\alpha_1,\beta_1,\alpha_2,\beta_2\) on \(X(1,b,c)\).

\begin{proposition}[Core curves on \(X(1,b,c)\) with slit length \(\ell\)]\label{prop:ExtCorePlumbing}
On \(X(1,b,c)\) one has the exact equalities
\[
\Ext_{X(1,b,c)}(\alpha_1)=\frac{1}{\,b-\ell\,},\qquad
\Ext_{X(1,b,c)}(\beta_1)=b,
\]
and
\[
\Ext_{X(1,b,c)}(\alpha_2)=\frac{1}{\,c-\ell\,},\qquad
\Ext_{X(1,b,c)}(\beta_2)=c.
\]
\end{proposition}

\begin{proof}
In the horizontal direction, the slit affects exactly the trajectories at heights \(0<y<\ell\).
Consequently, the closed horizontal trajectories supported in the \(T_1\) side are precisely those in the
embedded horizontal strip
\[
C_{\alpha_1}\ :=\ \bigl([0,1]\times(\ell,b)\bigr)/\{(0,y)\sim(1,y)\}\ \subset\ X(1,b,c),
\]
which is a Euclidean cylinder of circumference \(1\) and height \(b-\ell\).  Its boundary consists of
(horizontal) saddle connections through the cone points at the slit endpoints, hence \(C_{\alpha_1}\) is
maximal among embedded Euclidean cylinders in the free homotopy class of \(\alpha_1\).
Therefore \(\alpha_1\) is a cylinder curve, and  for the cylinder \(C_{\alpha_1}\) we have
\(\Ext_{X(1,b,c)}(\alpha_1)=1/(b-\ell)\).
The same argument on the \(T_2\) side gives \(\Ext_{X(1,b,c)}(\alpha_2)=1/(c-\ell)\).

\smallskip

The vertical slit is parallel to the vertical foliation.  In particular, the vertical core class
\(\beta_1\) is realized by the full vertical Euclidean cylinder on the \(T_1\) side, of circumference \(b\)
and height \(1\).  Hence we have: \(\Ext(\beta_1)=b\), and similarly \(\Ext(\beta_2)=c\).
\end{proof}

\medskip
\noindent
\textbf{The admissible curves and the two extremal--length assignments.}\;
Set
\[
\mathcal A_{}:=\{\alpha_1,\alpha_2\},
\qquad
\sigma_\ast=\id_{\mathcal A_{}}.
\]
Thus the same curve family \(\gamma\in\mathcal A_{}\) is compared across the two sides.
Define the extremal--length assignments (Definition~\ref{def:ExtAssign}) by
\[
\Ext_{\mathrm{I}}(u;\gamma):=\Ext_{X_{\mathrm{I}}(u)}(\gamma),
\qquad
\Ext_{\mathrm{II}}(u;\gamma):=\Ext_{X_{\mathrm{II}}(u)}(\gamma).
\]
Since \(X_{\mathrm{II}}(u)\) is marked by \(f\circ J\), we have
\(
\Ext_{X_{\mathrm{II}}(u)}(\alpha_i)=\Ext_{X(u)}(J(\alpha_i))=\Ext_{X(u)}(\beta_i).
\)

\begin{proposition}\label{prop:mismatch_db}
For \(u=(b,c)\in\Delta^{\mathrm{sc}}_{}\) one has
\[
m_{\alpha_1}(u)=\log\frac{\Ext_{\mathrm{I}}(u;\alpha_1)}{\Ext_{\mathrm{II}}(u;\alpha_1)}
=\log\frac{\Ext_{X(u)}(\alpha_1)}{\Ext_{X(u)}(\beta_1)}
=\log\frac{(1/(b-\ell))}{b}
=-\log\bigl(b(b-\ell)\bigr),
\]
\[
m_{\alpha_2}(u)=\log\frac{\Ext_{\mathrm{I}}(u;\alpha_2)}{\Ext_{\mathrm{II}}(u;\alpha_2)}
=\log\frac{\Ext_{X(u)}(\alpha_2)}{\Ext_{X(u)}(\beta_2)}
=\log\frac{(1/(c-\ell))}{c}
=-\log\bigl(c(c-\ell)\bigr).
\]
Consequently,
\[
H(u)=m_{\alpha_1}(u)^2+m_{\alpha_2}(u)^2
=\bigl(\log(b(b-\ell))\bigr)^2+\bigl(\log(c(c-\ell))\bigr)^2.
\]
\end{proposition}

\begin{proof}
This is immediate from Proposition~\ref{prop:ExtCorePlumbing} together with the observation above that the
\(J\)--marked side exchanges \(\alpha_i\) with \(\beta_i\) in extremal length.
\end{proof}

\subsubsection{Verification of the hypotheses}

We now verify the Hypothesis from Section~\ref{subsec:C_axioms} for this example, so that
Theorem~\ref{thm:main} applies on \(\Delta^{\mathrm{sc}}_{}\).

\begin{theorem}\label{thm:good_example}
With \(\mathcal A_{}\), \(\sigma_\ast=\id\), and the extremal--length assignments above, the
dumbbell domain \(\Delta^{\mathrm{sc}}_{}\cong\R_{>0}^2\) is Teichm\"uller regular
(Definition~\ref{def:TeichReg}), degeneration--detecting (Definition~\ref{def:DegDetect}), and pushable
(Definition~\ref{def:Pushable}).  Hence \(\Delta^{\mathrm{sc}}_{}\) contains a reflexive point.
In fact the unique reflexive point is \((b,c)=\bigl(b_\ell,b_\ell\bigr)\), where \(b_\ell> \ell\) is the unique positive root of \(b(b-\ell)=1\), i.e.\ \(b_\ell=\frac{\ell+\sqrt{\ell^2+4}}{2}\).
\end{theorem}

\begin{proof}
First we check: admissibility of \(\mathcal A_{}\).
The curves \(\alpha_1,\alpha_2\) are represented by simple closed geodesics on the translation surface
\(X(1,b,c)\), hence are essential and simple on \(S_2\).  Taking \(\sigma_\ast=\id\) gives an admissible curve
set in the sense of Definition~\ref{def:AC}.

\smallskip
\noindent

By Proposition~\ref{prop:ExtCorePlumbing},
\[
\log \Ext_{\mathrm{I}}(u;\alpha_1)=-\log(b-\ell),\qquad \log \Ext_{\mathrm{I}}(u;\alpha_2)=-\log(c-\ell),
\]
and
\[
\log \Ext_{\mathrm{II}}(u;\alpha_1)=\log b,\qquad \log \Ext_{\mathrm{II}}(u;\alpha_2)=\log c.
\]
Thus the extremal--length coordinate maps
\[
E_{\mathrm{I, II}}:\Delta^{\mathrm{sc}}_{}\longrightarrow \R^{\mathcal A_{}},
\qquad
E_{\mathrm{I}}(b,c)=(-\log(b-\ell),-\log(c-\ell)),\quad
E_{\mathrm{II}}(b,c)=(\log b,\log c)
\]
are \(C^\infty\) and have invertible Jacobian matrices \(\mathrm{diag}(-1/(b-\ell),-1/(c-\ell))\) and
\(\mathrm{diag}(1/b,1/c)\), respectively.
Hence they are local immersions, proving Teichm\"uller regularity.

\smallskip
\noindent

If \((b_k,c_k)\) leaves every compact subset of \(\Delta^{\mathrm{sc}}_{}=\{(b,c)\in\R_{>0}^2:\ \ell<\min\{b,c\}\}\),
then either \(b_k\longrightarrow \ell^+\) or \(b_k\longrightarrow\infty\), or \(c_k\longrightarrow \ell^+\) or \(c_k\longrightarrow\infty\).
By Proposition~\ref{prop:mismatch_db}, \(\bigl|\log(b_k(b_k-\ell))\bigr|\longrightarrow\infty\) or \(\bigl|\log(c_k(c_k-\ell))\bigr|\longrightarrow\infty\),
hence \(H(b_k,c_k)\longrightarrow\infty\).  Thus the height detects degeneration.

\smallskip
\noindent
Last thing to check pushability. Let
\[
V_{\alpha_1}:=b\,\frac{\partial}{\partial b},
\qquad
V_{\alpha_2}:=c\,\frac{\partial}{\partial c}.
\]
Since \(m_{\alpha_1}(b,c)=-\log\bigl(b(b-\ell)\bigr)=-\log b-\log(b-\ell)\), we compute
\[
(dm_{\alpha_1})_u(V_{\alpha_1})=b\,\frac{\partial}{\partial b}\Bigl(-\log b-\log(b-\ell)\Bigr)
= -\Bigl(1+\frac{b}{b-\ell}\Bigr)\ <\ 0
\quad\text{whenever }m_{\alpha_1}(u)\neq 0,
\]
and similarly
\[
(dm_{\alpha_2})_u(V_{\alpha_2})=
 -\Bigl(1+\frac{c}{c-\ell}\Bigr)\ <\ 0
\quad\text{whenever }m_{\alpha_2}(u)\neq 0.
\]
Moreover, the cross--derivatives vanish.  Thus the pushability conditions
(Definition~\ref{def:Pushable}) hold with incidence sets \(I(\alpha_1)=\{\alpha_1\}\) and \(I(\alpha_2)=\{\alpha_2\}\).

\smallskip
Having verified the hypotheses, Theorem~\ref{thm:main} implies the existence of a reflexive point.
In the present example we can identify it explicitly from the formula in Proposition~\ref{prop:mismatch_db}:
\[
H(b,c)=\bigl(\log(b(b-\ell))\bigr)^2+\bigl(\log(c(c-\ell))\bigr)^2.
\]
Thus \(H(b,c)=0\) if and only if \(b(b-\ell)=1\) and \(c(c-\ell)=1\), i.e.\ \(b=c=b_\ell=\frac{\ell+\sqrt{\ell^2+4}}{2}\).
\end{proof}

The two restricted characters in this example arise from the same flat data \((X(1,b,c),dz)\) but with two
different markings: the ``$II$'' marking is precomposed by \(J\), which exchanges horizontal and vertical
directions on each torus side.
At the reflexive point \((b,c)=(b_\ell,b_\ell)\), the extremal--length equalities imply that the marked surfaces
\(X_{\mathrm{I}}(b_\ell,b_\ell)\) and \(X_{\mathrm{II}}(b_\ell,b_\ell)\) coincide in Teichm\"uller space; equivalently, \(J\) is realized by a conformal
automorphism of the underlying surface.
This is exactly the mechanism by which one underlying Riemann surface is recovered from the restricted pair
of characters (cf.\ Corollary~\ref{cor:twoforms}).

\subsection{A configuration where no finite edge--parallel admissible curve system is sufficiently rich}\label{subsec:bad_example}

We now describe a family of translation surfaces for which the edge--parallel admissible curves
fail to produce Teichm\"uller--regular extremal--length coordinates, so the  argument of theorems cannot work.

The essential feature is that every edge--parallel admissible curve represents the same free
homotopy class.  Although the parameter domain has real dimension \(\ge 2\), the extremal length of
every such admissible curve depends only on a single scalar combination of the parameters (in fact,
after scale--fixing it becomes constant).  Consequently, no finite family of these admissible curves can
give the immersion property required for Teichm\"uller regularity
(Definition~\ref{def:TeichReg}); the height minimization scheme therefore breaks down at the first step.

\subsubsection{From a surface family to a parameter domain}
Consider a translation surface \(Y\) whose horizontal foliation consists of three Euclidean cylinders
\[
A_1,\ A_2,\ A_3
\]
stacked vertically, each of horizontal circumference \(w>0\), with heights \(h_1,h_2,h_3>0\).
Concretely, one may model \(A_i\) by the rectangle \([0,w]\times (y_{i-1},y_i)\) (where \(y_i-y_{i-1}=h_i\))
with the vertical sides identified by translation.
Glue the top boundary of \(A_1\) to the bottom boundary of \(A_2\) by translation with twist parameter
\(t_1\in[0,w)\), and glue the top boundary of \(A_2\) to the bottom boundary of \(A_3\) by translation with
twist parameter \(t_2\in[0,w)\).
Finally, glue the remaining boundary components so that the resulting surface has genus \(2\).

All three cylinders have the same core curve free homotopy class; denote it by \(\alpha\).  We pass to a scale--fixed slice by imposing
\[
w=1,
\qquad
h_1+h_2+h_3=1.
\]
Then \((h_1,h_2,t_1,t_2)\) vary in a domain of real dimension at least \(3\) (even after dividing by the
obvious residual symmetry given by global horizontal translation).
In particular, the parameter space is not locally one--dimensional.

\subsubsection{Attempting to choose $\mathcal A_C$: all edge--parallel cycles coincide}

\begin{lemma}\label{lem:only_one_class}
In the above parallel--cylinder surface, every edge--parallel admissible cycle is freely homotopic
to \(\alpha\) (or \(\alpha^{-1}\)).
\end{lemma}

\begin{proof}
By definition, any edge--parallel admissible cycle \(\gamma\) is obtained by concatenating horizontal arcs
inside the cylinders \(A_i\) with vertical arcs along their boundary seams. Hence \(\gamma\subset A:=A_1\cup
A_2\cup A_3\). Topologically \(A\) is an annulus (a cylinder), and any essential simple closed curve in an
annulus is freely homotopic to the core curve. Since \(\gamma\) is primitive and simple, it follows that
\(\gamma\simeq \alpha\) or \(\alpha^{-1}\).
\end{proof}

Now we compute the extremal length of \(\alpha\) and hence of every edge--parallel admissible cycle.

\begin{proposition}\label{prop:rank_one_failure}
Let \(\gamma\) be any edge--parallel admissible cycle on \(Y\). Then
\[
\Ext_Y(\gamma)=\Ext_Y(\alpha)=\frac{w}{h_1+h_2+h_3}.
\]
In particular, after the scale normalization \(w=1\) and \(h_1+h_2+h_3=1\), the extremal length
\(\Ext_Y(\gamma)\) is constant, independent of \((h_1,h_2,t_1,t_2)\).
\end{proposition}

\begin{proof}
By Lemma~\ref{lem:only_one_class}, \(\Ext_Y(\gamma)=\Ext_Y(\alpha)\). Set \(H:=h_1+h_2+h_3\) and
\(A:=A_1\cup A_2\cup A_3\). As a flat surface, \(A\) is a Euclidean cylinder of circumference \(w\) and height
\(H\); the twist parameters \(t_1,t_2\) only change boundary identifications by translations and do not change
its conformal modulus. Hence \(\mathrm{mod}(A)=H/w\), so
\[
\Ext_Y(\alpha)=\Ext_A(\alpha)=\frac{1}{\mathrm{mod}(A)}=\frac{w}{H}=\frac{w}{h_1+h_2+h_3}.
\]
On the slice \(w=1\) and \(H=1\) this equals \(1\), hence is constant.
\end{proof}

\subsubsection{Hypothesis failure: no Teichm\"uller regularity from edge--parallel data}

\begin{theorem}[Failure of Teichm\"uller regularity for edge--parallel admissible curves]\label{thm:bad_example}
In the parallel--cylinder configuration above, there is no choice of finite edge--parallel admissible
curve set \(\mathcal A\) for which the extremal--length coordinate map
\[
\mathbf{E}:\quad (h_1,h_2,t_1,t_2)\longmapsto \bigl(\log\Ext_Y(\gamma)\bigr)_{\gamma\in\mathcal A}
\]
is a local immersion. Consequently the axiom of Teichm\"uller regularity (Definition~\ref{def:TeichReg})
cannot be verified using only such admissible curves.
\end{theorem}

\begin{proof}
Fix any finite edge--parallel admissible set \(\mathcal A\). By Proposition~\ref{prop:rank_one_failure},
for every \(\gamma\in\mathcal A\),
\[
\log\Ext_Y(\gamma)=\log\!\Bigl(\frac{w}{h_1+h_2+h_3}\Bigr).
\]
Thus \(\mathbf{E}\) factors through the single scalar function \((h_1,h_2,t_1,t_2)\mapsto w/(h_1+h_2+h_3)\).
On the scale--fixed slice \(w=1\) and \(h_1+h_2+h_3=1\), this scalar is constant, so \(\mathbf{E}\) is locally
constant and hence \(d\mathbf{E}\equiv 0\). Therefore \(\mathbf{E}\) cannot be a local immersion.
\end{proof}

In particular, edge--parallel extremal--length data has (at most) rank \(1\) here and cannot detect the
higher--dimensional Teichm\"uller variation of \(Y\); to obtain Teichm\"uller regularity one must enlarge the
admissible family beyond the edge--parallel class.

\end{document}